\documentclass[12pt,letterpaper]{amsart}
\usepackage[pass]{geometry}
\usepackage{amssymb,amsmath,amsthm,amsgen}
\usepackage{comment}
\usepackage{}
\usepackage{ytableau}

\newcommand{\old}[1]{}
\theoremstyle{plain}
\newtheorem{thm}{Theorem}[section]
\newtheorem{lem}[thm]{Lemma}

\newtheorem{cor}[thm]{Corollary}

\newtheorem{prop}[thm]{Proposition}
\theoremstyle{definition}
\newtheorem{defn}[thm]{Definition}

\numberwithin{equation}{section}
\numberwithin{equation}{thm}

\def\a{{\alpha}}

 at 10truept

\title{On conjugacy classes of $S_n$ containing all irreducibles}

\author{Sheila Sundaram}
\address{Pierrepont School, One Sylvan Road North, Westport, CT 06880}
\email{shsund@comcast.net}
\date{3 February 2016; revised  28 September 2016, 27 February 2017}

\dedicatory{For Priyanka}


\begin{document}

\begin{abstract} It is shown that for the conjugation action of the symmetric group $S_n,$ when $n=6$ or $n\geq 8,$ all $S_n$-irreducibles appear as constituents of a single conjugacy class, namely, one indexed by a partition $\lambda$ of $n$ with at least two parts, whose parts are all distinct  and taken from the set of odd primes and 1.  The following simple characterisation of conjugacy classes containing all irreducibles is proved: If $n\neq 4,8,$ the partition $\lambda$ of $n$ indexes a global conjugacy class for $S_n$ if and only if it has at least two parts, and all its parts are odd and distinct. 
\end{abstract}

\maketitle

\section{Introduction}
Consider the conjugation action of $S_n$ on itself.  Its orbits are the conjugacy classes, indexed by partitions $\lambda$ of $n.$ 
 In this paper we address the following question: Is there a single conjugacy class which contains all the irreducibles for the conjugation action of $S_n$?  This question was answered affirmatively for the alternating group $A_n$ and for other simple sporadic groups in \cite[Theorem 1.6]{HZ}, where such a conjugacy class is called a {\it global}  class.    Our result is the following (see also Theorem 5.1): 

\begin{thm} For $n\geq 2,$ there is a global conjugacy class for the conjugation action of $S_n$ if and only if  $n=6$ or $n\geq 8.$   More precisely,   let $\lambda$ be a partition of $n$ into distinct parts, such that all parts greater than 1 are odd primes, and such that $\lambda$ has at least two parts.  In this case,  every $S_n$-irreducible occurs in  the conjugacy class indexed by  $\lambda,$  provided $n=6$ or $n\geq 9.$ If $n=8,$ every $S_8$-irreducible occurs in the class indexed by the partition $(7,1).$
\end{thm}

It follows immediately that (for $n\geq 8$) the permutation module arising from the conjugation action of $S_n$ on itself, as well as the twisted analogue of this action studied in \cite{Su2}, both contain a copy of every $S_n$-irreducible.  For the conjugacy action,  proofs that  this fact holds for all $n\geq 3$ were given by A. Frumkin, D. Passman and T. Scharf respectively, in \cite{F}, \cite[Theorem 1.10]{P} and \cite{Sch}, in each case  by a different method; a more general result was proved in \cite{Su2}.  For the sign-twisted conjugacy action the analogous result was shown to hold for all $n$ in \cite[Theorem 4.9]{Su2}. 

A related question is Passman's problem of determining the kernel of the adjoint representation of the group algebra.  In \cite{P}, Passman proves that the conjugacy action for $S_n$ contains every irreducible by showing  the equivalent statement \cite[Lemma 1.12]{P},  that the kernel of the adjoint action is trivial. Heide and Zalesski (\cite{HZ}; see also \cite{HSTZ}) conjecture that for a finite simple group, this kernel is trivial if and only if the group $G$ admits a  global conjugacy class (see [3], p.165, remarks leading to Conjecture 1.5). 

Theorem 1.1 strengthens and gives an alternative proof of  \cite[Theorem 4.17]{Su2} that the $S_n$-conjugacy action on the subset of even permutations contains every irreducible.

Our proof of Theorem 1.1 uses the description of the ordinary and twisted conjugacy action as the symmetric or exterior power of the conjugation action of $S_n$ on an $n$-cycle. This formulation  was  exploited in \cite{Su2} to derive many properties of these and other related representations.  In the next section we summarise the key facts from \cite{Su2} that are used in this paper, referring to \cite{M} for background on symmetric functions. 

\section{Preliminaries}

Consider the representation $1\uparrow_{C_n}^{S_n}$ of $S_n$ obtained by inducing the trivial representation of a cyclic subgroup $C_n$ of order $n.$  Clearly this is the representation afforded by the conjugation action on the class of $n$-cycles.
 Let $f_n$ denote the Frobenius characteristic (see \cite{M}) of the induced representation $1\uparrow_{C_n}^{S_n}$ of $S_n$ obtained by inducing the trivial representation of a cyclic subgroup $C_n$ of order $n.$  Let $h_n$ and $e_n$ denote respectively the homogeneous and elementary symmetric functions  of degree $n,$  and let $[\, ]$ denote the plethysm operation.   Then it is clear that  conjugation on the conjugacy class indexed by a partition $\lambda$ of $n$ with $m_i$ parts equal to $i,$ has  Frobenius characteristic  $H_\lambda[F]=\prod_{i\geq 1} h_{m_i}[f_i].$    Similarly, the Frobenius characteristic of the twisted conjugation action on the class indexed by $\lambda$ is $E_\lambda[F]=\prod_{i\geq 1} e_{m_i}[f_i].$ 

In particular, when $\lambda$ consists of distinct parts $\lambda_i,$ the Frobenius characteristic of both the ordinary and twisted conjugation actions on the conjugacy class indexed by $\lambda$ is simply the product of symmetric functions $\prod_i f_{\lambda_i}.$  A crucial property of the representations $f_n$, stated in 
Proposition 3.1, and proved in \cite{Su2},  allows us to determine when the representation corresponding to these products contains every $S_n$-irreducible. 

The crux of our arguments therefore lies in a detailed analysis of the products $\prod_i f_{\lambda_i}$. Since the Schur function $s_\lambda$ is the Frobenius characteristic of the irreducible indexed by $\lambda, $ we are led to products of Schur functions and the Littlewood-Richardson rule.   Recall \cite{M} that a {\it lattice permutation} of weight $\nu$ is a sequence of positive integers with the number $i$ appearing $\nu_i$ times,  such that in any initial segment of the sequence, the number of $i$'s is never less than the number of $(i+1)$'s.  For partitions $\lambda, \mu, \nu,$ the multiplicity of the Schur function $s_\lambda$ in the product $s_\mu s_\nu,$ is the Littlewood-Richardson coefficient (or LR-coefficient)  $c_{\mu, \nu}^{\lambda}$ (\cite{M}).
This coefficient counts the number of semi-standard fillings of weight $\nu$  of  the skew-shape $\lambda/\mu,$ i.e., one that is weakly increasing in the rows, left to right, and strictly increasing down the columns, with the property that  when this filling is read right to left along rows, and top to bottom down the (skew-)shape, the resulting sequence is a lattice permutation. 

We record the following  basic facts (\cite{M})  to summarise the key properties of the Littlewood-Richardson coefficient that we will need:

\begin{prop}  The following are equivalent for partitions $\lambda, \mu,\nu:$ 
\begin{enumerate}
\item $s_\lambda$ appears in the product $s_\mu s_\nu;$  
\item $s_\nu$ appears in $s_{\lambda/\mu};$
\item $   c_{\mu, \nu}^{\lambda}\geq 1 ; $
\item  there is  a semi-standard tableau of weight $\nu$ of  the skew-shape $\lambda/\mu,$ such that when the tableau is read right to left along rows and top to bottom down the skew-shape, the resulting sequence is a lattice permutation.
\end{enumerate}
Furthermore, given partitions $\lambda$ and $\mu,$ there is a partition $\nu$ satisfying one of the above equivalent conditions if and only if $ \mu\subset \lambda .$
\end{prop} 

\section{The ingredients of the proof}

The characterisation of a global $S_n$-class as given in Theorem 1.1 is motivated by the following key property  regarding the irreducibles appearing in  $f_n$.  This property was established in \cite{Su2} using the well-known character values of the representation $f_n.$

\begin{prop}\cite[Lemma 4.7]{Su2} If $k$ is an odd prime then the coefficient of $s_\mu$ in the Schur function expansion of $f_k$ is a  positive  integer for every partition $\mu$ of $k$ except $\mu=(k-1,1)$ and $\mu=(2, 1^{k-2}).$ 
\end{prop}

We will need several technical lemmas, which we state in terms of a class of functions whose definition is motivated by Proposition 3.1.

\begin{defn}   Define symmetric functions  $A_n$ and $g_n$ 
($n\geq 0$), as follows:
 
\begin{center}$A_n=\sum_{\mu\vdash n} s_\mu$ for $n\geq 1,$\quad  $g_n=\sum_{\mu\vdash n, \mu, \mu^t\neq (n-1,1)} s_\mu$  for $n\ge 3.$ 
\end{center}

  \noindent 
Also define $g_2=A_2=s_{(2)}+s_{(1^2)},$ $g_1=A_1=s_{(1)},$ 
and $g_0=A_0=1.$  Thus $g_n=A_n -(s_{(n-1,1)} +s_{(2, 1^{n-2})} )$ for $n\geq 4,$ while $g_3=A_3-s_{(2,1)}.$
\end{defn}

\begin{lem} Let $q\geq 5$ and $n\geq q+1,$ or $q=3,4$ and $n\geq q+2.$ Let $\lambda$ be any  partition of $n.$  Then $\lambda$ contains a partition $\bar\mu$ of $q$ 
different from $(q-1,1)$ and $(2, 1^{q-2}).$  
\end{lem}

\begin{proof} Certainly $\lambda$ contains {\it some} partition of $q;$ suppose first that it contains $(q-1,1).$ 

First suppose $\lambda_1\geq\lambda_2\geq 2,$ i.e., the Ferrers diagram of $\lambda$ contains a 2 by 2 square.  If $q\geq 4,$ we can take $\bar\mu=(q-2,2).$   If $q=3,$ then since $n\geq 5,$ there is at least one additional cell in $\lambda.$ If this cell is in row 1 or row 2, we may take $\bar\mu=(3).$ Otherwise 
we may take $\bar\mu=(1^3).$  

Now suppose the Ferrers diagram of $\lambda$ does not contain a 2 by 2 square, so it must be a hook $(r, 1^{n-r}),$ where $q-1\leq r\leq n-1.$   Take $\bar\mu=(q)$ if $r\geq q.$ Otherwise $r=q-1,$ and $n-r=n-q+1\geq 2,$ so we may take $\bar\mu=(q-2, 1^2).$  This poses a problem only if $q= 4.$  But then $\lambda=(3, 1^{n-3}),$ and since $n\geq q+2=6,$  we can take $\bar\mu=(1^4).$

The case when $\lambda$ contains $\mu=(2,1^{q-2})$ is easily treated by applying the previous argument to the transpose or conjugate shape $\lambda^t.$   \end{proof}

\begin{lem} Let $q\geq 5$ and $n\geq q+1,$ or $q=3,4$ and $n\geq q+ 2, $  or $q=1,2$ and $n\geq q.$ Then for every partition $\lambda$ of $n,$ the Schur function $s_\lambda$ appears in the product $A_{n-q} g_q.$ Equivalently, there are partitions $\mu$ of $n-q$ and $\nu$ of $q,$ $\nu\notin\{(q-1,1), (2, 1^{q-2})\}$ for $q\geq 3$, such that the LR-coefficient $c^\lambda_{\mu, \nu}$ is positive.  In particular, this holds for $A_1g_q=g_1g_q$ for all $q\neq 3,4,$ and for $A_2 g_q$ for all $q.$
\end{lem}

\begin{proof}  From Lemma 3.3 (for $q\geq 3$), or by the definition of $g_q$ (for $q=1,2$), we know that there is a partition $\nu$ of $q$ contained in $\lambda$  such that $s_\nu$ appears as a summand of $g_q.$ By Proposition 2.1 there is a partition $\mu$ of $n-q$ such that the LR-coefficient 
$c^\lambda_{\mu, \nu}$ is nonzero, i.e., such that $s_\lambda$ appears in the product $s_\mu s_\nu.$  By the definition of  $A_{n-q}$, the result follows.    
 \end{proof}
Note that the conditions in the hypothesis are tight: $A_1g_4 =g_4 s_{(1)}$ does not contain the Schur function corresponding to the hook $(3,1^2),$ and $A_1g_3=g_3s_{(1)} $ does not contain the Schur function indexed by $(2,2).$

The next lemma is the heart of this paper.  We defer the proof to the next section, because  the arguments are technical and unilluminating beyond their immediate context.

\begin{lem} Let $p>q\geq 4,$ or $p\geq 6, q=3,$ or $p=4, q=3.$ Then in the Schur function expansion of $g_p g_q,$  the Schur function $s_\lambda$ appears with positive coefficient for every partition $\lambda$ of $p+q.$  Equivalently, there are partitions $\mu$ of $p,$ $\mu\notin \{(p-1,1), (2, 1^{p-2})\} $, and $\nu$ of $q,$ $\nu\notin\{(q-1,1), (2, 1^{q-2})\}$, such that the LR-coefficient $c^\lambda_{\mu, \nu}$ is positive. 
\end{lem}

Again, the conditions in the hypothesis are tight: one checks by direct computation that $s_\lambda$ does not appear in $g_5 g_3$ for $\lambda= (4^2), (2^4)$, and does not  appear in $g_4^2$ for $\lambda=(6,1^2), (3,1^5), (4,2,1^2)$ and $(3^2,2).$ However $g_4 g_3$ does contain all the irreducibles; this is checked by direct computation. Finally $g_3^2$ does not contain $s_\lambda$ for $\lambda=(3,2,1).$


We now state conditions, one necessary and two sufficient,  for a conjugacy class to be global.

\begin{prop}  Let $\lambda$ be a partition of $n.$
\begin{enumerate}
\item  If $\lambda$  indexes a global conjugacy class, then all its parts must be odd and distinct, and it must have at least two parts.
\item
If $n-1$ is a prime greater than or equal to 5,  the class indexed by the partition $(n-1,1)$ is a global conjugacy class for $S_n.$ 
\item If $\lambda$ is a partition not containing 1 which indexes a global conjugacy class for $S_n,$ then the partition $\lambda\cup \{1\}$ indexes a global conjugacy class for $S_{n+1}.$
\end{enumerate}
These facts also hold for the twisted conjugacy action of $S_n.$
\end{prop}
\begin{proof}   For Part (1),  recall from Section 2 that the Frobenius characteristic of the conjugacy action on the class indexed by $\lambda$ is the product of plethysms $\prod_i h_{m_i}[f_i]$ where the part $i$ occurs $m_i$ times in $\lambda.$ By restricting the sign representation to an $i$-cycle, it follows easily by reciprocity that if $i$ is even then $f_i$ does not contain $s_{(1^i)}. $ 
 One way to see this is to invoke Proposition 2.1, which shows that the sign representation $s_{(1^n)}$ occurs in a product of two Schur functions if and only if the latter are both equal to the sign representation.  Thus if $m_i\geq 2$ or $i$ is even, the sign representation cannot occur in $h_{m_i}[f_i]$ and  the claim follows.   For the twisted conjugacy action, one checks that the sign representation occurs in $e_{m_i}[f_i]$ if and only if $i$ is odd, and the trivial representation occurs if and only if $m_i=1$; hence the claim.   Part (2) follows from Proposition 3.1 and  Lemma 3.4, with $q=n-1,$ since $n\geq 6$ and the representation $f_{n-1}$ contains the representation $g_{n-1}.$   For Part (3), first note that if $F_\lambda$ denotes the representation on the conjugacy class indexed by $\lambda,$  and 1 is NOT a part of $\lambda,$ then the representation on the class $\lambda\cup\{1\}$ is obtained by inducing up and thus its Frobenius characteristic is $F_\lambda\cdot f_1.$ Now the claim follows by reciprocity:
 it suffices to observe, (using the notation of Definition 3.2), that $F_\lambda$ contains $A_n$ if $\lambda$ is a global class, since $A_n$ contains all irreducibles,  and  the multiplicity of $s_\mu$ in  $A_n\cdot f_1$ equals the inner product $\langle A_n, s_{\mu/(1)}\rangle,$ which is nonzero by definition of  $A_n$. \end{proof}


In order to apply Proposition 3.1 to the proof of Theorem 1.1, we will need a number-theoretic result on the representation of an integer as a sum of primes due to  R. Dressler.  We state Dressler's result as a proposition.

\begin{prop} \cite{D} Every positive integer $n\neq 1,2,4,6,9$ can be written  as a sum of distinct odd primes. 
\end{prop}

\begin{cor}  Every even integer $n\geq 8$ can be written as a sum of at least two distinct odd primes, and every odd integer $n\geq 9$ can be written as a sum of 1 and  at least two distinct odd primes.
\end{cor}

\begin{proof} The first statement is simply Dressler's result above.  For odd $n\geq 9$ it suffices to apply the theorem  to the even integer $n-1,$ which must be  a sum of at least two distinct odd primes.   \end{proof}

Recall that $f_n$ denotes the Frobenius characteristic of the representation $1\uparrow_{C_n}^{S_n}$ of $S_n,$ where $C_n$ is a  cyclic subgroup  of order $n,$   i.e., of the conjugacy action on $n$-cycles.  

We now have all the ingredients necessary to complete the 

\vskip .1in

\noindent
{\bf Proof of Theorem 1.1:} Recall that if $\lambda$ consists of distinct parts $\lambda_i,$ the Frobenius characteristic of both the ordinary and twisted conjugation action on the conjugacy class indexed by $\lambda$ is simply the product $\prod_i f_{\lambda_i}.$

 Let $n\geq 10.$ By Corollary 3.8, there is  a partition  $\lambda$ of $n$ with distinct parts  taken from the set containing 1 and the odd primes, such that  the number of odd parts greater than 1 is at least two. From Proposition 3.1 we know that, as representations,  $f_{\lambda_i}$ contains $g_{\lambda_i}$ for every part $\lambda_i$ of $\lambda.$ 

Lemma 3.5 tells us when the product of two $g_k$'s will contain all irreducibles, and Lemma 3.4 tells us when a product of the form $A_{n-k} g_k$ will contain all irreducibles.
We will apply Lemma 3.5 once, and then  repeatedly invoke Lemma 3.4. For the two largest parts, we have $p=\lambda_1>q=\lambda_2\geq 3.$   Invoking Lemma 3.5, (noting that if $q=3$ then $p\geq 7$), we see that $s_\mu$ appears in the product $f_p f_q$ for every partition $\mu$ of $p+q.$  Equivalently, 
$f_p f_q$ contains $A_{p+q}$ as representations. 
Denote the  remaining smaller parts of $\lambda$ (if any)  by $\mu_1> \ldots > \mu_r\geq 1$; i.e., $\lambda=\{p,q\}\cup \mu.$
Consider the product $F_k=f_p f_q \left(\prod_{ i=1}^k f_{\mu_i}\right),$ for $k=1, \ldots, r.$  We claim that, as a representation,  $F_k$ contains all irreducibles.   Because all integers involved are distinct, $p\geq q+1$ and $q\geq \mu_1+1$,  so  $p+q\geq 2\mu_1+3>\mu_1+2.$  Hence we can invoke Lemma 3.4  to conclude that $F_1=(f_p f_q)\cdot  f_{\mu_1}$  contains all irreducibles, since it contains the product $A_{p+q} g_{\mu_1},$ i.e., it  contains $A_{p+q+\mu_1}$ as a representation.  Now iterate Lemma 3.4. For the inductive step note that $p+q+\sum_{i=1}^k \mu_i >\mu_{k+1}+2$ for all $k<r.$ We conclude that $F_k$ contains all irreducibles. 

When $n<10,$ 
one has, from Proposition 3.6 (3) and Lemma 3.4: 
\begin{enumerate}
\item For each of the  cases $n=6,8,9$, the following partitions respectively index a unique global conjugacy class:  $(5,1), (7,1), (5,3,1).$ (See the remarks following Lemma 3.5:  the product $g_5g_3$ does not contain all irreducibles.)
\item For $2\leq n\leq 5$ and $n=7,$ there is no global conjugacy class.
\end{enumerate}
This finishes the proof of Theorem 1.1. \qed

The above argument in fact establishes the global property of containing every irreducible for a general class of representations.
Let $W_n$ be the $S_n$-module whose Frobenius characteristic is $g_n.$ For any partition $\lambda$ of $n,$ let $W_\lambda$ be the $S_n$-module whose Frobenius characteristic is the product $\prod_i g_{\lambda_i}.$  It is  immediate from Lemmas 3.4 and 3.5 that $W_\lambda$ contains all irreducibles in each of the cases below.  
\begin{enumerate} 
\item  $\lambda=(1^n),$ or 
  $\lambda$ contains the part 1 and also a part $k$ for some $k\neq 3,4$  or 
 \item $\lambda$ contains both 3 and 4 as parts, or
 2 is a part of $\lambda,$ or 
\item  $\lambda$ has at least two {\bf distinct} parts $p>q\geq 4,$ or $p\geq 6, q=3.$
\end{enumerate}
(For Part (3)  use Lemma 3.5 and the remark following it to conclude that for the smallest pair of parts $s>t,$ we have that $g_s g_t$ contains $A_{s+t}$, and $s+t\geq 7;$ now iterate Lemma 3.4 as above.) 
%
Checking that $g_1 g_3^2, g_3^3$ and $ g_4^3$ contain all irreducibles, while $g_1g_4^2$ lacks $s_{(3^3)},$ we have:
\begin{thm} Every irreducible appears in $W_\lambda$  for all $\lambda$ with at least two parts {\bf except} the following six partitions:
$\lambda=(3,1), (4,1), (3,3), (4,4), (5,3), (4,4,1).$ 
\end{thm}

\section{Proof of  Lemma 3.5}

In this section we prove Lemma 3.5 via a case-by-case analysis. The case $p=4, q=3$ is checked by direct computation of the Schur function expansion of $g_4g_3,$ so we will confine ourselves to the other cases.

\begin{proof} Let $\lambda$ be any partition of $p+q.$ If there are partitions $\mu$ of $p$ and $\nu$ of $q$ such that $\mu\notin \{(p-1,1), (2, 1^{p-2})\} $ and $\nu\notin \{(q-1,1), (2, 1^{q-2})\}$ and $s_\lambda$ appears in the Schur function expansion of the product 
$s_\mu\cdot s_\nu,$ then $s_\lambda$ appears with positive coefficient in the product $g_p g_q.$  Our goal is to show that such partitions $\mu,\nu$ always exist, using Proposition 2.1.

We may assume $\lambda\notin \{(n-1,1), (2, 1^{n-2})\},$  by observing that 
 $s_{(n-1,1)}$ appears in the product $s_{(p)} s_{(q)}$ and 
$s_{(2, 1^{n-2})}$ appears in the product $s_{(1^p)} s_{(1^q)}.$ 

By Lemma 3.3, there is a  $\mu\subset \lambda$ such that $s_\mu$ appears in $g_p.$ Hence there is at least one partition $\nu$ of $q$ such that   $s_\lambda$ appears in the Schur function expansion of the product $g_p s_\nu.$ 

We need to take care of the cases $\nu=(q-1,1), (2, 1^{q-2}).$ The case $\nu=(2, 1^{q-2})$ for $\lambda$ follows from the conjugate case $\nu=(q-1,1)$ for $\lambda^t,$ since the sum $g_k$ is self-conjugate. Thus it suffices to consider $\nu=(q-1,1)$.  
  By analysing the possible configurations of $\lambda/\mu,$ we show that it is possible to find a $\bar\nu$ which is neither $(q-1,1)$ nor its transpose, and a partition $\bar\mu$ of $p$ such that $\bar\mu$  is neither $(p-1,1)$ nor its transpose,
 for which $s_\lambda$ appears in the product $s_{\bar\mu} s_{\bar\nu}.$  Equivalently, 
the LR-coefficient $c_{\bar\mu, \bar\nu}^{\lambda}$ is positive.

Suppose $\lambda/\mu$ is a horizontal strip, i.e., no two squares occur in the same column.  In this case clearly the skew-shape can be filled with all 1's, which is a lattice permutation, and hence 
$c^\lambda_{\mu, (q)}\neq 0.$ 

Henceforth we assume $\lambda/\mu$ is NOT a horizontal strip. 
We will treat the cases $q\geq 4$ and $q=3$ separately.  First let $q\geq 4.$  We shall further distinguish between a connected shape $\lambda/\mu$ (Cases 1-3 below) and a disconnected one (Cases 4-6 below).

We consider first the case when the skew-shape is connected (and $q\geq 4$). 
 The most general connected configuration admitting a lattice permutation of weight $(q-1,1)$ is as in Figure 1 below, where the skew-shape is connected, occupies two consecutive rows,  and has exactly one pair of squares appearing in the same column.  Let us call the inner corner cell $d.$ If $d$ appears in row $i$ and column $j$ of the Ferrers diagram of $\lambda,$ let  $c_1$ be the cell in row $i+1$ and column $ j-1,$ and let $c_2$ be the cell in row $i$ and column $ j+1.$  Note that at least one of these two cells must exist in $\lambda,$  in order to accommodate a lattice permutation of weight $(q-1,1), $ since $q\geq 4.$

\vskip .1in
\ytableausetup{centertableaux}
\begin{center}
\begin{tiny}
\begin{ytableau}
X &X  &X   &X &X &X &X   &X &X &X &X \\
X &X  &X &X &d &c_2 &\ldots   &\ldots &\ldots \\
 X &X  & \ldots  &c_1 &\ldots   \\ 
X  & X \\
\end{ytableau}\end{tiny}
\vskip.08in
\begin{small} \vbox{Figure 1:  $\lambda/\mu$  is NOT a horizontal strip, and admits \\a lattice permutation filling of weight $(q-1,1)$}\end{small} \end{center}
%
\begin{description} 
\item[{\bf Case 1}] Assume cells $c_1$ and $c_2$ both exist in $\lambda. $ Fill the first row of the skew-shape with all 1's. For the second row of the skew-shape,  put 2's in  the cell below $d$ and in cell  $c_1,$ and 1's for the remaining cells to the left of $c_1.$ This gives a lattice permutation, since there are at least two 1's preceding the 2's: both cell $d$ and cell $c_2$ are filled with 1's.  It has weight $(q-2, 2),$ and hence for $ \bar\nu=(q-2,2),$ we have $c_{\mu, \bar\nu}^\lambda\neq 0,$ and thus $s_\lambda$ appears in the product $s_\mu\cdot s_{\bar\nu},$  hence in $g_p\cdot g_q.$  

\item[{\bf Case 2}] Cell $c_2$ does not exist in $\lambda,$ so that we have the \lq\lq backwards L\rq\rq configuration of Figure 2a. 

\vskip .1in
\ytableausetup{centertableaux}
\begin{center}
\begin{tiny}
\begin{ytableau}
X &X  &X  &X  &X & X &X &X   &X &X &X &{ X} \\
X &X  &X &X &{\bf X'} &d  \\
 X &X  &{\bf y} & \ldots  &c_1 &   \\ 
X  & X \\
\end{ytableau}
\end{tiny} \begin{small} Figure 2a \end{small}\end{center}
\vskip .1in

In this case the cell in row $i+1$ and column $j$ (immediately below $d$) is a  cell on the south-east boundary of $\lambda.$ Let $X'$ be the cell in row $i$ and column $j-1,$ (i.e., immediately to the left of cell $d$) and let $y$ be the cell of $\lambda$ in row $i+1$ which is immediately adjacent to $\mu.$  Thus $y$ is a cell of $\lambda/\mu$, while $X'$ is a cell of $\mu.$ Note that $y\neq c_1$ since $\lambda/\mu$ has size $q\geq 4.$ Consider the partition $\bar\mu=\mu\cup \{y\}\backslash\{X'\},$ still a partition of $p.$ 
Then $\lambda/\bar\mu$ consists of a 2 by 2 square on the south-east boundary of $\lambda$ together with $(q-4)$ squares in row $(i+1).$ It can therefore be filled with a lattice permutation of weight $(q-2,2)$ (see Figure 2b). 

\vskip .1in
\ytableausetup{centertableaux}
\begin{center}
\begin{tiny}
\begin{ytableau}
X &X  &X &X   &X &X &X &X   &X &X &X &X \\
X &X  &X &X &{\bf 1} &{\bf 1}  \\
 X  &X &{\bf 1}  & {\bf 1}  &{\bf 2} &{\bf 2}   \\ 
X  &X\\
\end{ytableau}
\end{tiny}
 \begin{small} Figure 2b \end{small} \end{center}
\vskip .1in
Now we need only check that $\bar\mu$ is not of shape $(p-1,1)$ or $(2, 1^{p-2}).$ In the former case, we would be forced to conclude that $\bar\mu$ has only $(q-2)$ squares (since the cell $y$ can be the only cell of $\bar\mu$ in row $i+1$), a contradiction since $|\bar\mu|=|\mu|=p>q.$  In the latter case, the cell occupied by $y$ is a boundary cell of $\lambda,$ so the only way this can happen is if $\bar\mu=(2,1),$ and thus the size of $\bar\mu$ is $3.$ But then $p=3$ and $q\geq 4,$ a contradiction to the hypothesis that $p>q.$   Thus $\bar\mu$ is the shape we seek, and $s_\lambda$ appears in the product $s_{\bar\mu}\cdot s_{(q-2,2)},$ hence in $g_p\cdot g_q.$

\item[{\bf Case 3}] Cell $c_1$ does not exist in $\lambda,$ so that we have the configuration in Figure 3. 

\vskip .1in
\ytableausetup{centertableaux}
\begin{center}
\begin{tiny}
\begin{ytableau}
X &X  &X   &X &X  &X  &X   &X  &X &X &X &X\\
X &X  &X   &X &X &X &X   &X &X &X &X \\
X &X  &X &X &d &c_2 &\ldots   &\ldots &\ldots \\
 X &X  &X  &X &   \\ 
X  & X \\
\end{ytableau}
\end{tiny}
 \begin{small} Figure 3 \end{small} \end{center}
\vskip .1in

Here we have four sub-cases to consider: Recall that $d$ is in row $i,$ column $j.$
\begin{description}
\item[{\bf Subcase 3a}] Suppose that either $i\geq 3$ and $j$ is arbitrary, or $i=2$ and $j\geq 2,$  and assume there is a corner cell $X'$ of $\lambda$ occupying a column $j'> j+q-2,$ i.e., a corner cell that is strictly to the right  of the right-most cell in the skew-shape.  
In particular, $X'$ will be strictly to the right of the cell $c_2.$ See Figure 3a. 

\vskip .1in
\ytableausetup{centertableaux}
\begin{center}
\begin{tiny}
\begin{ytableau}
X &X  &X   &X &X  &X  &X   &X  &X &X &X &X\\
X &X  &X   &X &X &X &X   &X &X &X &{\bf X'} \\
X &X  &X &X &d & c_2 &\ldots   &\ldots &\ldots \\
 X &X  &X  &X &   \\ 
X  & X \\
\end{ytableau}
\end{tiny}
 \begin{small} Figure 3a \end{small} \end{center}
\vskip .1in

Take $\bar\mu=\mu\cup\{d\}\backslash\{X'\}.$  
 Then  $\bar\mu$ must contain a 2 by 2 square (this is clear if $i\geq 3;$ if $i=2,$ then since $j\geq 2,$ the cell $d$ is the bottom right corner of the 2 by 2 square.)  Thus  $\bar\mu$ cannot be one of the forbidden shapes,  and $\lambda/\bar\mu$ is a   horizontal strip.  
Hence $c^\lambda_{\bar\mu, (q)}\neq 0$ and $s_\lambda$ appears in the product $s_{\bar\mu}\cdot s_{(q)}.$ 
\item[{\bf Subcase 3b}] Again consider the case when either $i\geq 3$ and $j$ is arbitrary, or $i=2$ and $j\geq 2,$  but now assume  there is no corner cell of $\lambda$ occupying a  column $j'> j+q-2.$ Rows $1, \ldots, i$ of $\lambda$ must therefore all have the same length.  Let $X'$ be the cell in row $(i-1)$ and column 
$j+q-2,$ immediately above the right-most cell of the 
skew-shape $\lambda/\mu.$  Note that as long as  $q\geq 4,$ the right-most cell of the skew-shape is not $c_2.$ Take $\bar\mu=\mu\cup\{d\}\backslash\{X'\}$ exactly as in Subcase 3a. 
%
The new skew-shape can clearly be filled with a lattice permutation of weight $(q-2,2):$ see Figures 3b and 3b'. Note that  $\bar\mu$ is a  shape appearing in $g_p$, because it contains a 2 by 2 square as in Subcase 3a.  

\vskip .1in
\ytableausetup{centertableaux}
\begin{center}
\begin{tiny}
\begin{ytableau}
X &X  &X   &X &X  &X  &X   &X  &X \\
X &X  &X   &X &X &X &X   &X &{\bf X'}  \\
X &X  &X &X &d &c_2 &\ldots   &\ldots &\ldots \\
 X &X  &X  &X &   \\ 
X  & X \\
\end{ytableau}  \end{tiny} \begin{small} Figure 3b \end{small} \end{center}
\vskip .2in

\ytableausetup{centertableaux}
\begin{center}\begin{tiny}
\begin{ytableau}
X &X  &X   &X &X  &X  &X   &X  &X \\
X &X  &X   &X &X &X &X   &X &{\bf 1}  \\
X &X  &X &X &d &{\bf 1} &{\bf 1}   &{\bf 1} &{\bf 2} \\
 X &X  &X  &X & {\bf 2}  \\ 
X  & X \\
\end{ytableau}\end{tiny} \begin{small} Figure 3b' \end{small}
 \end{center}
\vskip .2in

\item[{\bf Subcase 3c}] Assume that $ i=2$ but $j=1;$ thus $\mu$ consists of only the first row of $\lambda,$ i.e., a  single row lying above the shape $(q-1,1),$ as in Figure 3c. 

\vskip .1in
\ytableausetup{centertableaux}
\begin{center}
\begin{tiny}
\begin{ytableau}
 X &X &X   &X &X &X &{\bf X''} &{\bf X'} \\
 d & c_2 &\ldots   &\ldots &\ldots \\
    \\ 
\end{ytableau}
\end{tiny}
 \begin{small} Figure 3c \end{small} \end{center}
\vskip .1in
 
 In this case, since $\mu$ has size  $p\geq (q-1)+2$, the single row of $\mu$ must exceed the next (second) row of $\lambda$ by at least 2  squares.  Let $X'$ and $X''$ be the two squares at the end of the first row, as in Figure 3c.  Taking $\bar\mu=\mu\cup\{d, c_2\}\backslash\{X', X''\},$ we see that $\lambda/\bar\mu$ is a horizontal strip, thus admitting a lattice permutation of weight $(q),$ and $\bar\mu$ has a 2 by 2 square. Hence in this case $s_\lambda$ appears in the product $s_{\bar\mu}\cdot s_{(q)},$ and $s_{\bar\mu}$ appears in $g_p.$
\item[{\bf Subcase 3d}] Now assume that  $i=1$ (and hence necessarily  $j\geq 2$),  i.e., the skew-shape spans the first two rows of $\lambda,$ as in Figure 3d.  

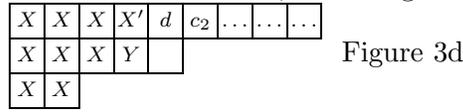
\ytableausetup{centertableaux}
\begin{center}\begin{tiny}
\begin{ytableau}
X &X  &X &X' &d &c_2 &\ldots   &\ldots &\ldots \\
 X &X  &X  &Y &   \\ 
X  & X \\
\end{ytableau}\end{tiny} \begin{small} Figure 3d \end{small}
  \end{center}
\vskip .1in 

Let $X'$ be the cell in row 1, column $j-1$ (i.e., immediately to the left of cell $d$) and let $Y$ be the  cell of $\mu$ just below $X'$. 

First suppose that  $X',Y$ are in column 1 of $\mu.$ See Figure 3d'.  

\vskip .1in
\ytableausetup{centertableaux}
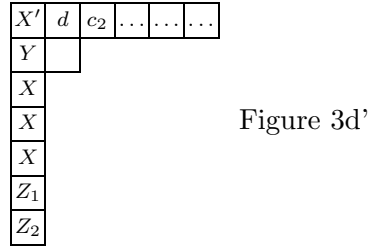
\begin{center}\begin{tiny}
\begin{ytableau}
X' &d &c_2 &\ldots   &\ldots &\ldots \\
Y &   \\ 
X \\
X\\
X\\
Z_1\\
Z_2\\
\end{ytableau}\end{tiny} \begin{small} Figure 3d' \end{small}
 \end{center}
\vskip .1in 

This means $\mu=( 1^p);$ since $p\geq q+1\geq 5,$ there are at least 3 cells in column 1 of $\mu$ below the  cell $Y.$  Let $Z_1, Z_2$ be the two bottom-most cells in $\mu.$
Let $\bar\mu=\mu \cup \{d,c_2\}\backslash\{Z_1, Z_2\}.$ It is clear that this produces a partition $\bar\mu$ of $p$ of shape $(3, 1^{p-3})$ whose Schur function appears in $g_p$ because $p-3\geq 2.$ 
 The skew-shape $\lambda/\bar\mu$ consists of a horizontal strip (from the first two rows) of size $q-2$, of which one square (in row 2) is disconnected, and  a connected vertical strip of size 2, i.e., one that will admit a lattice permutation of weight $(q-2, 2).$  Hence $c^\lambda_{\bar\mu, (q-2,2)}\neq 0,$ and $s_\lambda$ appears in the product $s_{\bar\mu}\cdot s_{(q-2,2)}.$   See Figure 3d''.
\vskip .1in
\ytableausetup{centertableaux}
\begin{center}\begin{tiny}
\begin{ytableau}
X' &X &X &{\bf 1}   &{\bf 1} &{\bf 1} \\
Y &{\bf 2}   \\ 
X \\
X\\
X\\
{\bf 1}\\
{\bf 2}\\
\end{ytableau}\end{tiny} \begin{small} Figure 3d'' \end{small}
  \end{center}

%
\end{description}
\end{description}
We have taken care of all cases in which $q\geq 4$   and $\lambda/\mu$ is a connected skew-shape.

Now assume that the skew-shape $\lambda/\mu$ is disconnected (and $q\geq 4$).  Since $\lambda/\mu$ admits a lattice permutation of weight $(q-1,1),$ it has  two or more connected components, and all except possibly one consist of a single row.  We may assume that the exceptional component is either a vertical strip of size 2, or the shape $(k,1)$ for some $k\geq 2,$ or the \lq\lq backwards $L$\rq\rq skew-shape shown in Figure 2a; otherwise it looks like the skew-shape in Figure 1, with both $c_1$ and $c_2$ being nonempty, and is thus taken care of locally by the argument of Case 1.
\begin{description}
\item[Case 4] We will begin  with the  \lq\lq backwards $L$\rq\rq case   shown in Figure 2a. 
If the size of the skew-shape in Figure 2a is 4 or greater, the argument in Case 2 applies; note that the outcome  is unaffected by the other connected components since they are horizontal strips (filling them with all 1's gives a permissible  lattice permutation).   Hence we need only consider the case when the skew-shape shown in Figure 2a  is of size exactly 3.
\begin{description}
\item[Subcase 4a] We have the skew-shape $\lambda/\mu$ in Figure 4a, where there is a connected horizontal strip to the right of and above (i.e., north-east of) the  \lq\lq backwards $L$\rq\rq component, whose cells are  labelled $d, a_1, c_1$.  In this case there is clearly a lattice permutation filling of weight containing a 2 by 2 square, (fill all the components that are horizontal strips with 1's, a 1 in cell $d$ and 2's in cells $c_1$ and $a$).  (Note that $q\geq 4.$)  See Figure 4a'.  Thus $s_\lambda$ appears in $g_p\cdot g_q.$
\vskip .1in
\ytableausetup{centertableaux}
\begin{center}\begin{tiny}
\begin{ytableau}
X &X  &X &X   &X &X  &X  &X   &X  &X &X &X &X\\
X &X  &X &X   &X &X &X &X   & & & & \\
X &X  &X &X &X &d  \\
 X  &X &X  &X  &c_1 &a   \\ 
X &X &X &X\\
X  &X\\
\end{ytableau}\end{tiny}
\vskip.08in
\begin{small}\vbox{Figure 4a: Horizontal strip north-east of  the skew-shape $d,a,c_1.$ }\end{small}
\end{center}

\ytableausetup{centertableaux}
\begin{center}\begin{tiny}
\begin{ytableau}
X &X  &X &X   &X &X  &X  &X   &X  &X &X &X &X\\
X &X  &X &X   &X &X &X &X   &{\bf 1} &\ldots &\ldots &{\bf 1} \\
X &X  &X &X &X &{\bf 1}  \\
 X  &X &X  &X  &{\bf 2} &{\bf 2}   \\ 
X &X &X &X\\
X  &X\\
\end{ytableau}\end{tiny}
\vskip.08in
\begin{small}\vbox{Figure 4a': Lattice permutation  with at least two 2's}
\end{small}\end{center}
\vskip .05in
\item[Subcase 4b] We have the skew-shape $\lambda/\mu$ in Figure 4b, where there is no  connected horizontal strip to the north-east  of  the \lq\lq backwards  L\rq\rq component, but there is (necessarily, because $q\geq 4$) ) a connected horizontal strip below (south-west of) it.   In this case again there is clearly a lattice permutation filling of weight $\nu$ containing a 2 by 2 square, obtained by filling the right-most cell of the {\it closest} connected horizontal strip with a 2 (and the remaining cells, if any, with 1's, including all other horizontal strips), hence a weight $\nu$ such that $s_\nu$ appears in $g_q.$ See Figure 4b'.
\vskip .05in
\ytableausetup{centertableaux}
\begin{center}\begin{tiny}
\begin{ytableau}
X &X  &X &X   &X &X &X &X   &X &X &X &X \\
X &X  &X &X &X &X &d  \\
 X  &X &X  &X  &X &c_1 &a   \\ 
X &X &X &X &X \\
X &X & & &\\
X  &X\\
\end{ytableau}\end{tiny}
\vskip.06in
\begin{small}Figure 4b: Horizontal strip south-west of  the skew-shape  $d,a,c_1.$ 
\end{small}\end{center}
\vskip .1in
\ytableausetup{centertableaux}
\begin{center}\begin{tiny}
\begin{ytableau}
X &X  &X &X   &X &X &X &X   &X &X &X &X \\
X &X  &X &X &X &X &{\bf 1}  \\
 X  &X &X  &X  &X &{\bf 1} &{\bf 2}   \\ 
X &X &X &X &X\\
X &X &{\bf 1} &{\bf 1} &{\bf 2}\\
X  &X\\
\end{ytableau}\end{tiny}
\vskip.06in
\begin{small}Figure 4b': Lattice permutation of weight $\nu$ with at least two 2's \end{small}
\end{center}

\end{description}
\item[Case 5] Now suppose that there are at least two connected components in $\lambda/\mu$, all but one of which are connected horizontal strips, and the exceptional component $L$ has shape $(k,1)$ for some $q-1>k\geq 2.$  We call $d_1,\ldots, d_k$ the cells (left to right) in the first row of $L,$ and $e_1$  the unique  cell in row 2 of $L$  immediately below $d_1.$   By assumption, there is at least one other connected component, which is a horizontal strip.  
 
\begin{description}
\item[Subcase 5a] If there is such a horizontal strip above $L,$  fill $L$ with the number 1 in cells $d_1, \ldots, d_{k-1}$ and put 2's in the cells  $e_1$  below $d_1,$ and in the right-most cell $d_k$.  Fill all other connected components with 1's. Clearly this gives a lattice permutation whose weight is $(q-2,2)$ and hence indexes a Schur function in $g_q.$ 
\item[Subcase 5b] If all the connected components of $\lambda/\mu$ that are horizontal strips are below $L$, fill $L$ with $k$ 1's in cells $d_1,\ldots , d_k$ and a 2 in cell $e_1.$  Let $C=c_1,\ldots, c_r$ be (left to right) the cells in the connected horizontal strip closest   to $L$ and below it.  Then $r\geq 1.$ Fill the right-most cell $c_r$ with a 2 (this is possible because $k\geq 2$), and all others with a 1.  Fill all other connected components with 1's. Once again  this gives a lattice permutation whose weight is $(q-2,2)$ and hence indexes a Schur function in $g_q.$ 
\end{description}
\item[Case 6] The final case to consider is when the exceptional  connected component  of $\lambda/\mu$ is a vertical strip $V$ of size 2, and all the others are horizontal strips (of size 1 or more).  There are four subcases: 
\begin{description}
\item[Subcase 6a] Let $q\geq 4.$ Assume that  there is a horizontal strip $H_1$ above $V$ and a horizontal strip $H_2$ below $V$, and assume these are the closest components to $V$. Fill $V$ with 1 and 2 in the obvious column-strict way.  Fill the left-most cell of $H_1$ with a 1, and the right-most cell of $H_2$ with a 2.  (These cells exist since $q\geq 4.$) Fill all other cells in the skew-shape $\lambda/\mu$ with 1's.  Clearly this gives a lattice permutation of weight $(q-2,2).$   See Figure 5a.

\vskip .1in
\ytableausetup{centertableaux}
\begin{center}\begin{tiny}
\begin{ytableau}
X &X  &X &X   &X &X  &X  &X   &X  &X &{\bf 1} &\ldots &{\bf 1}\\ 
X &X  &X &X   &X &X &X &X   &X &X  \\
X &X  &X &X &X &X &{\bf 1}  \\
 X  &X &X  &X  &X &X &{\bf 2}   \\ 
X &X &X &X &X &X\\
X &X &{\bf 1} &\ldots &{\bf 1} &{\bf 2}\\
X  &X\\
\end{ytableau}\end{tiny} \begin{small} Figure 5a \end{small}
 \end{center}
\vskip .1in
\item[Subcase 6b]  First let $q\geq 5.$ Assume that all the horizontal strips of the skew-shape $\lambda/\mu$ are  above $V$. Fill all the horizontal strips  with 1's, and $V$ with 2 and 3. Thus we have a lattice permutation filling of weight $(q-2,1,1). $ Since   $q-2\geq 3,$  this indexes a Schur function in $g_q.$ See Figure 5b. This argument fails if $q=4;$ it will be addressed in {\bf Subcase 6d}.
\vskip .2in
\ytableausetup{centertableaux}
\begin{center}\begin{tiny}
\begin{ytableau}
X &X  &X &X   &X &X  &X  &X   &X  &X &{\bf 1} &\ldots &{\bf 1}\\ 
X &X  &X &X   &X &X &X &X   &{\bf 1}&{\bf 1}  \\
X &X  &X &X &X &X &{\bf 2}  \\
 X  &X &X  &X  &X &X &{\bf 3}   \\ 
X &X &X &X &X &X\\
X  &X\\
\end{ytableau}\end{tiny}\begin{small} Figure 5b \end{small}
 \end{center}
\vskip .1in
\item[Subcase 6c] Let $q\geq 5.$  Now assume all the horizontal strips of the skew-shape $\lambda/\mu$ lie below $V$, and let $H_0$ be the closest one below $V$. Fill $V$ with 1 and 2, and fill the right-most cell of $H_0$ with a 3, and all remaining cells in the skew-shape with 1's.  Again we have a lattice permutation filling of weight $(q-2,1,1). $ Since   $q-2\geq 3,$  this indexes a Schur function in $g_q.$ See Figure 5c. Again, the argument fails if $q=4,$ and will be addressed in {\bf Subcase 6d} below.
\vskip .1in
\ytableausetup{centertableaux}
\begin{center}\begin{tiny}
\begin{ytableau}
X &X  &X &X   &X &X &X &X   &X   &X  &X  \\
X &X  &X &X &X &X &{\bf 1}  \\
 X  &X &X  &X  &X &X &{\bf 2}   \\ 
X &X &X &X &X &X\\
X &X &{\bf 1} &\ldots &{\bf 1} &{\bf 3}\\
X  &{\bf 1} \\
X\\
\end{ytableau}\end{tiny}\begin{small} Figure 5c \end{small}
 \end{center}
\vskip .1in
\item[Subcase 6d] Now let $q=4,$ so that in the skew-shape we have, in addition to the vertical strip $V,$ either two disconnected single squares or a single connected horizontal strip of size 2. In the former case the entire skew-shape is clearly a vertical strip of size 4, so can be filled with a lattice permutation of weight $(1^4),$ and the question is settled.  In the latter case, our skew-shape $\lambda/\mu$ consists of a vertical strip of size 2 and a horizontal strip of size 2. 
Suppose that the top cell of $V$ is in row $\alpha$ and column $j,$ while the left-most cell (call it $d$) of $H$ is in row $i$ and column $\beta.$ 
We are interested in the portion $S(\mu)$ of $\mu$ lying between rows $\alpha$ and $i,$ and columns $j$ and $\beta.$ Let $X$ be a cell of $\mu$ in the south-east boundary of $S(\mu)$. 

First let the horizontal strip $H$ lie below the vertical strip $V,$ so $i\geq \alpha+2$ and $j\geq \beta+2.$
We replace $\mu$ with the partition $\bar\mu=\mu\backslash X\cup \{d\}.$  The resulting skew-shape $\lambda/\bar\mu$ consists of the vertical strip $V$ and two other cells; the higher new cell, $X,$ is in a row above $H$ and a column to the left of $V.$  In all cases the following filling gives a lattice permutation of weight $(2^2)$: put 1 and 2 in $V,$ 1 in $X$ and 2 in the right-most cell of $H.$ 
It remains to check that $\bar\mu$ is not a forbidden shape.  In all except one case $\bar\mu$ contains either three rows or three columns.   The exceptional case is  when $i=\alpha+2$ and $j=\beta+2,$ forcing the cells of $\mu$ lying between rows $\alpha$ and $i,$ and columns $j$ and $\beta,$ to form a 2 by 2 square.  But $\mu$ has $p\geq 5$ cells, so in $\mu,$ there is either a row (of length at least 3) above $V$ or a column (of length at least 3) to the left of $H,$ and the difficulty is resolved.

Next assume  that the horizontal strip $H$ is above the vertical strip $V,$ so that $i<\alpha$ and $j<\beta.$  Here we have three subcases to consider: 
\begin{description}
\item[Subcase 6d(i)]
 Assume that 
there is a cell $X$ in the south-east boundary of  $\lambda$ in row $i'$ and column $j'$, where $i<i'<\alpha, j<j'<\beta,$  i.e., $i'$ is  a row index  strictly above  $V$ and strictly below the row of $H$, and $j'$ is a column index also strictly to the right of $V$ and strictly to the left of $H.$ Recall that $d$ is the left-most cell of $H$.  Note that the skew-shape $\lambda/\mu$ is the union of $V$ and $H.$ Set $\bar\mu=\mu\cup\{d\}\backslash \{X\}.$ One checks easily that 
$\bar\mu$ is not of shape $(p-1,1)$ or $(2, 1^{p-2}),$ because  (a) $p\geq 5$ and (b) $X$ must be part of a 2 by 2 square of $\mu$, so including the cell $d$ in $\mu$ results in a row of length at least 3.  This ensures that $\bar\mu$ contains the shape $(3,1,1)$. Also now $\lambda/\bar\mu$ is clearly a vertical strip.    
\item[Subcase 6d(ii)] If no such cell $X$ exists, this means the strips $H$ and $V$ are at the ends of a hook shape $(r, 1^s)$ for some $r,s$ such that 
$1\leq r+s\leq p.$  Consider first the case $r+s=p.$ In this case $\lambda$ is also a hook.
In Figure 5d below,  $\mu$ is a hook $(\mu_1, 1^{p-\mu_1})$ with the vertical strip $V$ at the bottom and the horizontal strip $H$ on top.   Also, by our initial hypothesis that $s_\mu$ appears in $g_p,$ we know $\mu_1\neq 2.$ 

\vskip .1in
\ytableausetup{centertableaux}
\begin{center} \begin{tiny}
\begin{ytableau}
X &X  &X &X    &X  &X &X & &\\
X  \\
X \\
 X    \\ 
X \\
  \\
  \\
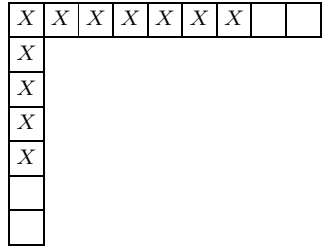
\end{ytableau}\end{tiny} \begin{small} Figure 5d: $\mu=(7, 1^4), \lambda=(9, 1^6)$ \end{small}
\end{center}
\vskip .1in

Suppose that $\mu_1=1.$ Then $\mu=(1^p)$ and $p\geq 5.$ If we take $\bar\mu$ to be the shape obtained by including the horizontal strip $H,$ and removing the bottom two squares of $\mu,$  then $\bar\mu=(3, 1^{p-3})$ and $\lambda/\bar\mu$ is a connected vertical strip of size 4, which satisfies our requirements.

Suppose that $\mu_1\geq 3,$ and let $\mu=(\mu_1, 1^{r}).$ (Thus $\mu_1+r=p\geq 5$ and $r\neq 1.$) The preceding construction now gives $\bar\mu=(\mu_1+2, 1^{r-2})$ and a skew-shape that is a vertical strip, as long as $r\geq 2$ and $ r\neq 3.$
The  difficulty arises when $r=3$ (the shape $\bar\mu$ does not appear in $g_p$)  or $r=0.$ In the former case $\mu=(\mu_1, 1^3),$ and we can take $\bar\mu=(\mu_1+1, 1^2)$: the skew-shape will still be a vertical strip with two connected components, one of size 1 and one of size 3. In the case $r=0,$ we can take $\bar\mu=(p-2, 1^2)$ and then the skew-shape will be a horizontal strip of size 4 from the first row.

\item[Subcase 6d(iii)] Now assume the strips $H$ and $V$ are at the ends of a nonempty hook shape $(r, 1^s)$ for some $r,s$ such that 
$1\leq r\leq r+s< p,$  so $\lambda$ is NOT a hook.  Since $\mu$ has size $p,$ we know that either the column to the left of $V$ is nonempty, or the row above $H$ is nonempty.  Assume the latter case; hence the row above $H$ contains at least $r+2\geq 3$ cells of $\mu$ (to ensure that $\lambda=\mu\cup V\cup H$ is a legal shape).     See Figure 5e. Hence  the right-most cell $Y$ in the row of $\mu$ above $H$ must lie strictly to the right of the left cell $d$ of $H.$ 
\vskip .1in
\ytableausetup{centertableaux}
\begin{center} \begin{tiny}
\begin{ytableau}
X &X  &X &X   &X &X  &X &X   &X  &X &X &Y \\
X &X  &X &X_1   &X_2 &\ldots &X_r  &d & \\
X &X  &X &X_1'  \\
 X  &X &X  &\ldots     \\ 
X  &X &X  &X_s'     \\ 
X &X &X &\\
X &X &X &\\
X  &X\\
\end{ytableau}\end{tiny} \begin{small} Figure 5e \end{small}
\end{center}
\vskip .1in

Set $\bar\mu=\mu\cup\{d\}\backslash\{Y\}. $ From the rows of $H$ and  $Y$, we see that since $r\geq 1,$  $\bar\mu$ contains a 2 by 2 square, and $\lambda/\bar\mu$ is now a vertical strip.  

Now assume that the row above $H$ is empty.  Then size considerations again imply that there is a nonempty column to the left of $V.$  By exchanging the top cell in $V$ with the bottom-most cell in the column to the left of $V,$ we arrive at the same conclusion as before.

\end{description}
\end{description}

\end{description}

   Hence we conclude that in all cases, $s_\lambda$ appears in $g_p g_q$  when $p>q\geq 4.$ 

Now let $q=3$ and $p=n-q\geq 6.$ Let $\lambda$ be a partition of $n$ and let $\mu\subset \lambda$ such that $s_\mu$ appears in $g_p=g_{n-3},$ and $\lambda/\mu$ can only be filled with a lattice permutation of weight $(2,1).$  We must show as before that there are partitions $\bar\mu$ and $\nu$ such that $s_{\bar\mu}$ appears in $g_{n-3},$ $s_\lambda$ appears in $s_{\bar\mu} \cdot s_\nu,$ and $\nu\neq (2,1).$ The cases are as follows:
\begin{enumerate}
\item Assume the skew-shape $\lambda/\mu$ forms a partition $(2,1)$, i.e., two cells with one cell below the left-most cell $d$ in the first row of the skew-shape.  
\begin{enumerate} 
\item If $\mu$ lies entirely above the skew-shape, let $X$ be the last cell in the row immediately above the skew-shape.  If $\mu$ has at least two rows,  then it is easy to see that we can take $\bar\mu=\mu\cup\{d\}\backslash \{X\}$ and the resulting skew-shape is now a vertical strip of size 3. Also, since $p\geq 6,$  $\bar\mu$ now spans three rows, and either contains a 2 by 2 square, or its first row has length at least 3, so the forbidden shapes are avoided.  Now suppose that $\mu$ consists of only one row lying above the skew-shape $(2,1).$ Let $X_1, X_2$ be the last two cells in the unique row of $\mu.$ Let $e$ be the cell below $d$ in the skew-shape.  Set $\bar\mu=\mu\cup\{d,e\}\backslash\{X_1, X_2\}.$ Then $\bar\mu=(\mu_1-2, 1,1)$ has first part equal to at least 4, and $\lambda/ \bar\mu$ is now a horizontal strip. 

 For the case when $\mu$ lies entirely to the left of the skew-shape, simply transpose this argument.
\item  Otherwise, at a minimum, we have cells both above and to the left of the skew-shape, at least 6 in all, so that, if $i, i+1$ are the rows of $\lambda,$ and  $j, j+1,$ the columns occupied by the skew-shape, $i,j\geq 2,$ then  $\mu_{i-1}\geq \mu_i+2, \mu_i=\mu_{i+1}\geq 1$ (and $\mu_{j-1}'\geq \mu_j'+2, \mu_j'=\mu_{j+1}'\geq 1$).  Again by a similar exchange of two cells as in (1) (a)  above, i.e.,  setting  $\bar\mu=(\mu_1,\ldots,\mu_{i-1}-1, \mu_i+1,\mu_{i+1}, \ldots)$ we arrive at the conclusion that the skew-shape is a vertical strip of size 3, and the shape $\bar\mu$ contains a 2 by 2 square.
\end{enumerate}
\item Now suppose the skew-shape is a \lq\lq backwards L\rq\rq  shape; refer to  Figure 6, where the skew-shape stops at the cell labelled $c_1.$   
\vskip .1in
\ytableausetup{centertableaux}
\begin{center}
\begin{tiny}
\begin{ytableau}
X &X  &X  &X  &X & X &X &X   &X &X &X &{\bf X_0} \\
X &X  &X &X &{\bf X'} &d  \\
 X &X   & \ldots  & X &c_1 &   \\ 
X  & X \\
\end{ytableau}
\end{tiny} \begin{small} Figure 6 \end{small}\end{center}
\vskip .1in
\begin{enumerate}
\item  If $\mu\backslash \{X'\}$ lies entirely above the three cells in the skew-shape, let $r$ be the lowest row of $\mu$ above the skew-shape.   Then $r\geq 2$ because $\mu\neq (p-1,1).$  In Figure 6, the cell labelled ${\bf X'}$ occurs in column 1 of $\lambda.$ Clearly $\mu_r\geq 2$ and $\mu_{r+1}=1,$ (${\bf X'}$ is the sole cell in row $r+1$ of $\mu$)  and $\mu$ has $r+1$ rows.  Then we may exchange the cell labelled $c_1$ in Figure 6 for the right-most cell labelled ${\bf X_0}$ in row $r$ of $\mu,$ producing $\bar\mu= (\mu_1, \ldots, \mu_{r}-1, \mu_{r+1}=1, 1).$   The lowest two cells of $\bar\mu$ are the cells labelled  ${\bf X'}$ and $c_1$ in Figure 6, so that $\bar\mu$ spans at least three rows. The condition $p\geq 6$ guarantees that if ${\bf X_0}$ is in the same column as $d,$ i.e., if $\mu_r=2,$  then either $\mu_{r-1}\geq 3$ or $r\geq 3;$ in the latter case, $\bar\mu$ has a 2 by 2 square.  Either way $\bar\mu$ is an allowed shape.  Clearly the new skew-shape is a vertical strip.

The case when $\mu\backslash \{X'\}$ lies entirely to the left of the three cells is treated by transposing the above argument.
\item There are cells above and to the left of the skew-shape in Figure 6. 
Referring to that figure, because $p\geq 6,$ at a minimum $\lambda$ consists of the 3 by 3 square in which the skew-shape occupies the lower right corner, and $\mu$ contains at least the remaining 6 cells in the 3 by 3 square. Again consider  the right-most cell ${\bf X_0}$ in the row immediately above the skew-shape (see Figure 6).  Exchange the cells labelled $c_1$ and  ${\bf X_0}.$ 
(In the extreme case ${\bf X_0}$ will be directly above the cell labelled $d).$  That is, $\bar\mu=\mu\cup\{c_1\}\backslash \{\bf X_0\}.$  Clearly $\bar\mu$ has a 2 by 2 square, and the new skew-shape is a vertical strip.  
\end{enumerate}  
\end{enumerate}
This completes our proof of Lemma 3.5 in the case $q=3,$ $p\geq 6.$
\end{proof}


\section{Final Remarks}

We conclude with one more theorem.  Proposition 5.2 below was conjectured by the author at the time of submission, and subsequently proved by Swanson \cite{Sw}. The work of the last two sections   shows that the truth of Proposition 5.2  implies Theorem 5.1.  This is immediate from Proposition 3.6 and the  technical Lemmas 3.4 and 3.5, as well as the remarks following each of them regarding the classes $(3,1)$ and $(5,3).$  
Equivalently, it is an immediate consequence of Theorem 3.9 that Proposition 5.2 implies the following:


\begin{thm} Let $n\neq 4, 8.$  The partition $\lambda$ of $n$ indexes a global conjugacy class for $S_n$ if and only if it has at least two parts, and all its parts are odd and distinct.  
\end{thm}
%
%
\begin{prop} (\cite{Sw}, conjectured in \cite[Remark 4.8]{Su2}) If $n$ is odd, $f_n$ contains all irreducibles except $(n-1,1)$ and $(2, 1^{n-2}).$  If $n$ is even, $f_n$ contains all irreducibles except $(n-1,1)$ and $(1^n).$
\end{prop}
\noindent
{\sc Acknowledgment:} \small{I am grateful to the anonymous referee for a thorough and meticulous reading, and for outlining a reorganisation of the paper which greatly improved the  exposition.}

\begin{thebibliography}{<99>}

\bibitem{D} R.~Dressler, \textit{A stronger Bertrand's Postulate with an application to partitions}, Proc. Amer. Math. Soc. \textbf{33} No. 2 (1972), 226-228.

\bibitem{F}  A.~Frumkin,
\textit{Theorem about the conjugacy representation of $S_n$}, Israel J. Math. \textbf{55} (1986), 
121--128.

\bibitem{HZ} G.~Heide and A.~E.~Zalesski, \textit{Passman's Problem on Adjoint Representations}, Contemporary Math. \textbf{420}  (2006), 163--176.

\bibitem{HSTZ} G.~Heide, J.~Saxl, P.~H.~ Tiep and A.~E.~Zalesski, \textit{Conjugacy action, induced representations and the Steinberg square for simple groups of Lie type}, J. London Math. Soc., \textbf{106} (2013), 908--930. 

\bibitem{M}  I.~G.~Macdonald, \textit{Symmetric functions and 
Hall polynomials}, Second Edition, Oxford University Press (1995).

\bibitem{P} D.~Passman, \textit{The adjoint representation of group algebras and enveloping algebras,}  Publicacions 
Matem\`atiques \textbf{36} (1992), 861--878.

\bibitem{Sch}  T.~Scharf,
\textit{Ein weiterer Beweis, da\ss\  die konjugierende Darstellung der symmetrischen Gruppe jede irreduzible Darstellung enth\"alt}, Arch. Math.  \textbf{54}  (1990), 427--429.

\bibitem{Su2} S.~Sundaram, \textit{The conjugacy action of $S_n$ and submodules induced by centralisers}, J. Alg. Comb., to appear. DOI: 10.1007/s10801-017-0796-9.
(arXiv:1603.0589v1).

\bibitem{Sw} J.~P. Swanson, \textit{On the existence of tableaux with given modular major index}, Algebraic Combinatorics, \textbf{ 1} (2018) no. 1,  3--21. DOI : 10.5802/alco.4. (arXiv:1701.04963v1).


\end{thebibliography}

\bibliographystyle{amsplain.bst}

\end{document}